\documentclass[12pt]{amsart}

\usepackage{amsmath,amssymb,amsthm, fullpage}
\usepackage{hyperref}

\usepackage{mathtools}
\usepackage{hhline, bm}
 \newcommand\house[1]{\mkern2mu\begin{array}{|@{\:}c@{\:}|}\hhline{|-|}#1\end{array}\mkern1mu}

\newtheorem{theorem}{Theorem}%[section]

\newtheorem{lemma}[theorem]{Lemma}

\newtheorem{proposition}[theorem]{Proposition}

\numberwithin{equation}{section}

\def\C{{\mathbb C}}
\def\D{{\mathbb D}}

\def\N{{\mathbb N}}

\def\R{{\mathbb R}}
\def\Z{{\mathbb Z}}

\def\T{{\mathbb T}}

\def\dis{\displaystyle}

\begin{document}

\title{House of algebraic integers symmetric about the unit circle}

\author{Igor E. Pritsker}
\address{Department of Mathematics, Oklahoma State University, Stillwater, OK 74078, U.S.A.}
\email{igor@math.okstate.edu}

\begin{abstract}
 We give a  Schinzel-Zassenhaus-type lower bound for the maximum modulus of roots of a monic integer polynomial
 with all roots symmetric with respect to the unit circle. Our results extend a recent work of Dimitrov, who
 proved the general Schinzel-Zassenhaus conjecture by using the P\'olya rationality theorem for a power series with
 integer coefficients, and some estimates for logarithmic capacity (transfinite diameter) of sets. We use an enhancement
 of P\'olya's result obtained by Robinson, which involves Laurent-type rational functions with small supremum norms,
 thereby replacing the logarithmic capacity with a smaller quantity. This smaller quantity is expressed via a weighted
 Chebyshev constant for the set associated with Dimitrov's function used in Robinson's rationality theorem. Our lower
 bound for the house confirms a conjecture of Boyd.
\end{abstract}


%\dedicatory{}

\maketitle

\section{Introduction and main result}

The subject of algebraic integers located near (or on) the unit circle $\T:=\{z\in\C:|z|=1\}$ is classical.
Kronecker \cite{Kr} proved that if an algebraic integer and all of its conjugates are located in the closed unit disk $\overline{\D}:=\{z\in\C:|z|\le 1\}$,
then it is either a root of unity or zero. For an algebraic integer $\alpha=\alpha_1$, with the complete set of conjugates $\{\alpha_k\}_{k=1}^n,$
the house of this algebraic integer is defined by
\[
\house{\alpha} := \max_{1\le k\le n} |\alpha_k|.
\]
The result of Kronecker may also be written in the form: If $\alpha$ is a non-zero algebraic integer that is not a root of unity, then
$\house{\alpha}>1.$ Another form of the same result can be recorded by using the Mahler measure of $\alpha$ defined by
\[
M(\alpha) := \prod_{k=1}^{n} \max(1,|\alpha_k|).
\]
Thus if $\alpha$ is a non-zero algebraic integer that is not a root of unity, then $M(\alpha)>1.$ Both versions indicate that any algebraic
integer that is not a root of unity must either be off the unit circle itself, or have conjugates off the unit circle. A natural question is how far away
those algebraic integers should be from the unit circle. This brings us to celebrated Lehmer's conjecture \cite{Le}, see also \cite{Bo}, \cite{EW},
\cite{Sm1}, \cite{Sm2} for more details and references. Lehmer observed from computations that the smallest Mahler measure of a non-zero and
non-cyclotomic algebraic integer seems to be coming from the largest (in absolute value sense) root $\alpha_L$ of the polynomial
$L(z) =  z^{10} + z^9 - z^7 - z^6 - z^5 - z^4 - z^3 + z + 1$. It turns out that all but two roots of Lehmer's polynomial are on the
unit circle, with the remaining roots being $\alpha_L>1$ and $1/\alpha_L.$  This prompted Lehmer to conjecture that any non-zero algebraic integer $\alpha$
that is not a root of unity must satisfy
\[
M(\alpha) \ge M(\alpha_L) = \alpha_L \approx 1.176280818.
\]
Note that $L(z)$ is a reciprocal polynomial, i.e., it satisfies $L(z)=z^{10}L(1/z).$ An algebraic integer of degree $n$ is called reciprocal if its minimal
polynomial $P$ is reciprocal, meaning $P(z)=z^n P(1/z).$ It is not difficult to see that all conjugates of a reciprocal algebraic integer are symmetric
with respect to the unit circle. If $\alpha$ is a root of a non-reciprocal irreducible polynomial with integer coefficients, Smyth \cite{Sm0} proved that
\[
M(\alpha) \ge \theta_0 \approx 1.3247,
\]
where $\theta_0$ is the positive root of $z^3-z-1.$ Lehmer's conjecture was proved for many other classes of algebraic integers, but the case of general
reciprocal $\alpha$ remains open. A related conjecture for the house of algebraic integer was made by Schinzel and Zassenhaus \cite{SZ}:
If $\alpha$ is a non-zero algebraic integer of degree $n$ that is not a root of unity, then
\begin{align} \label{SZ}
\house{\alpha} \ge 1+c/n
\end{align}
for an absolute constant $c>0.$ Note that Lehmer's conjecture implies that of Schinzel and Zassenhaus by the inequality
\[
\house{\alpha} \ge M(\alpha)^{1/n} > 1+\frac{\log M(\alpha)}{n}.
\]
Dimitrov recently proved the conjecture of Schinzel and Zassenhaus by showing that
\begin{align} \label{Dimgen}
\house{\alpha} \ge 2^{\frac{1}{4n}} > 1+\frac{\log 2}{4n},
\end{align}
see Theorem 1 of \cite{Di}. However, questions on the optimal values of $c$ in \eqref{SZ} for specific
classes of algebraic integers remain open. The latter questions were raised by Boyd \cite{Boy} on the bases of computations, see Conjectures (A)-(D)
in his paper. In particular, Boyd conjectured that the optimal (largest possible) lower bounds for the house are all coming from non-reciprocal
algebraic integers, in contrast to Lehmer's conjecture. Moreover, if $n$ is divisible by $3$, then Boyd conjectured that \eqref{SZ} holds with
\begin{align}\label{BC}
c = \frac{3}{2} \log \theta_0 \approx 0.4217.
\end{align}
The best bounds for non-reciprocal algebraic integers are due to Dubickas \cite{Du1} and \cite{Du2}, who showed that \eqref{SZ} holds with
\[
c \approx  0.30965.
\]
Dimitrov proved that for reciprocal algebraic integers with all conjugates off the unit circle \eqref{SZ} holds with
\begin{align} \label{Dimsym}
c = \frac{\log 2}{2} \approx  0.34657,
\end{align}
see Theorem 6 in \cite{Di}. We improve this result to $c \approx 0.44068,$  and thereby confirm Boyd's conjectures in the special case when
$\alpha$ is a non-zero reciprocal algebraic integer with all conjugates outside the unit circle.

\begin{theorem} \label{Bconj}
If $\alpha$ is a reciprocal algebraic integer of degree $n$, with complete set of conjugates $\{\alpha_k\}_{k=1}^n \bigcap \T = \emptyset,$ then
\begin{align}\label{c-conj}
\house{\alpha} \ge (1+\sqrt{2})^{\frac{1}{2n}} > 1 + \frac{\log(1+\sqrt{2})}{2n}.
\end{align}
\end{theorem}

It is clear that the degree $n$ must be even in the settings of the above theorem as one half of conjugates are located inside $\T$ and the other
half is outside $\T$ due to symmetry.

We give an outline of proof of Theorem \ref{Bconj} in the next section. Some technical results from potential theory that are necessary for the
proof are established in Section 3. A complete proof of the main result is contained in Section 4.

\section{Essential ideas of the proof} \label{Outline}

We first discuss a sketch of Dimitrov's proof for \eqref{Dimgen}. As the non-reciprocal case is known (cf. \cite{Du1}), we let $\alpha$ be a reciprocal
algebraic integer with the complete set of conjugates $\{\alpha_k\}_{k=1}^n,$ where $\alpha=\alpha_1$, and with the minimal polynomial
\[
P(z)=\prod_{k=1}^n (z-\alpha_k).
\]
Since $P$ is reciprocal, we have that $P(0) = 1.$ Define the auxiliary polynomials
\begin{align*}
  P_2(z)=\prod_{k=1}^n (z-\alpha_k^2) \quad \mbox{and} \quad P_4(z)=\prod_{k=1}^n (z-\alpha_k^4),
\end{align*}
and note that $P_2,P_4\in\Z[z].$ The arithmetic information was captured by Dimitrov in the function
\begin{align} \label{Dimfun}
D(z) := \sqrt{P_2(z) P_4(z)}.
\end{align}
We state the following result, condensed from Proposition 2.2 and Lemma 2.3 of \cite{Di}.

\begin{proposition} \label{Dimprop}
The Maclaurin series of $D(z)$ has integer coefficients. Assuming that $P_2$ is not a perfect square, we have that
$P$ is cyclotomic if and only if $D(z)$ is rational.
\end{proposition}

Rationality of $D(z)$ is deduced from the well known results due to P\'olya, see the original papers \cite{Po1} and \cite{Po2}, and also
\cite{Ro} for further history and discussion. Consider the function $f(z):=D(1/z).$ It is clear that the Laurent series expansion of $f$ near
$\infty$ consists of powers of $1/z$, and has integer coefficients by Proposition \ref{Dimprop}. This function can be defined as analytic in
$\overline{\C}\setminus K,$ where $K:=\bigcup_{1\le k\le n} ([0,\alpha_k^2]\cup [0,\alpha_k^4]),$ by introducing
cuts from the origin to the zeros of $P_2$ and $P_4$ in order to define appropriate branches of complex square root.
P\'olya's theorem states that if the transfinite diameter of $K$ is less than one, meaning that this set is sufficiently
small, then $f$ must be a rational function, hence $D$ is so too, and hence $P$ is cyclotomic. Complete discussions of
transfinite diameter (identical to logarithmic capacity and Chebyshev constant) may be found in \cite{Ra} and \cite{Ts}. Using the simple fact
that the transfinite diameter is increasing with the set (cf. Theorem 5.1.2 of \cite[p. 128]{Ra}), we enlarge $K$ to $\tilde{K}$
by extending each segment of $K$ to the length $\house{\alpha}^4$. A much deeper result of Dubinin, see Corollary 4.7 of
\cite[p. 118]{Dub}, states that the transfinite diameter of $\tilde{K}$ is largest when all segments of $\tilde{K}$ are equally
spaced in the angular sense. The transfinite diameter of this equally spaced configuration of $2n$ segments is known to be
$(\house{\alpha}^{8n}/4)^{1/(2n)}=\house{\alpha}^4/2^{1/n}$ by Theorem 5.2.5 and Corollary 5.2.4 of \cite[p. 134]{Ra},
where we applied the mapping $z^{2n}$ that transforms $\tilde{K}$ into a segment emanating from the origin of length
$\house{\alpha}^{8n}$, whose transfinite diameter is $\house{\alpha}^{8n}/4.$
It follows from the above argument that if $\house{\alpha}^4/2^{1/n}<1$ then $P$ is cyclotomic. Thus if $P$ is not
cyclotomic, then the opposite inequality holds, and we arrive at \eqref{Dimgen}.

The proof of Theorem \ref{Bconj} follows a similar scheme. We also use Proposition \ref{Dimprop} in the same fashion, but slightly
modify the domain of $D(z)$. Due to symmetry of the set $\{\alpha_k\}_{k=1}^n$ with respect to the unit circle, we can assume
that all conjugates come in symmetric pairs $\alpha_k$ and $\alpha_{n-k+1}$, with $\alpha_k=1/\bar{\alpha}_{n-k+1}$, for $k=1,\ldots,n/2.$
We set
\begin{align}\label{fun}
  F(z) := \prod_{k=1}^{n/2} \sqrt{(z-\alpha_k^2)(z-\alpha_{n-k+1}^2)}
  \prod_{k=1}^{n/2} \sqrt{(z-\alpha_k^4)(z-\alpha_{n-k+1}^4)},
\end{align}
where $\sqrt{(z-\alpha_k^2)(z-\alpha_{n-k+1}^2)}$ is defined as holomorphic in $\C\setminus[\alpha_k^2,\alpha_{n-k+1}^2]$
by selecting a single valued branch of the root that is asymptotic to $z$ at $\infty$, and the same approach is applied to
the second product of roots as well. Thus $F(z)$ is analytic in $\C\setminus E,$ where
\begin{align}\label{E}
E := \bigcup_{1\le k\le n/2} \left([\alpha_k^2,\alpha_{n-k+1}^2] \cup [\alpha_k^4,\alpha_{n-k+1}^4]\right).
\end{align}
Observe that $E$ is a proper subset of  $K$, so that $E$ has smaller transfinite diameter than that of $K.$ Another advantage of
this construction is that we can now use two (Laurent) series expansions of $F$: one about $\infty$, as in Dimitrov's proof, and another one
about the origin. In a neighborhood of the origin, we have $F(z)=D(z)$ so that the Maclaurin series of $F$ has integer coefficients by
Proposition \ref{Dimprop}. For the Laurent series at $\infty$, we use the identity $F(z) = z^n F(1/z),\ z\in\C\setminus E,$ which is
verified in Lemma \ref{expansions}, and conclude that this expansion also has integer coefficients. This enables us to use the following
result of Robinson that enhances the rationality theorem of P\'olya, see \cite[p. 533]{Ro}.

\begin{theorem} \label{Rob}
Suppose that $F(z)$ is analytic in a domain $G$ that contains both $0$ and $\infty$, and that $F(z)$ has Laurent expansions with integer
coefficients of the form:
\[
F(z)=\sum_{k=0}^\infty a_k z^{-k} \ \mbox{near $\infty$} \quad \mbox{and} \quad F(z)=\sum_{k=0}^\infty b_k z^k \ \mbox{near $0$}.
\]
If there is a Laurent-type rational function with complex coefficients of the form
\[
h(z) = \sum_{k=-l}^m A_k z^k, \quad \mbox{with $|A_m|\ge 1$ and $|A_{-l}|\ge 1$,}
\]
such that $|h(z)|<1$ for $z\in E:=\C\setminus G$, then $F(z)$ is rational.
\end{theorem}

It is clear from \eqref{fun} that rationality of $F$ implies rationality of $D$, hence implies in turn that $P$ is cyclotomic by Proposition \ref{Dimprop} as before.
Robinson's theorem essentially replaces the transfinite diameter of $E$ used in P\'olya's theorem, which is equal to the Chebyshev constant of $E$ defined via monic
polynomials with the least supremum norms on $E$ (cf. \cite{Ra} and \cite{Ts}), by a smaller quantity defined via the supremum norms
of  Laurent-type rational functions $h(z).$ We may assume that $h(z)$ has equal number of positive and negative powers due to symmetry of our
set $E$ defined in \eqref{E} with respect to $\T,$ and write
\[
|h(z)| = \left|\sum_{k=-m}^m A_k z^k\right| = |z|^{-m} \left|\sum_{k=0}^{2m} A_{k-m} z^k\right| = w(z)^{2m} |Q_{2m}(z)|.
\]
Here, we introduce the weight function $w(z):=|z|^{-1/2}$ and consider weighted polynomials of the form $w^{2m} Q_{2m}$, where $\deg(Q_{2m}) = {2m}.$
Such weighted polynomials were studied in detail by the methods of potential theory in \cite{ST}, see Chapter III in particular. The relevant quantity we need
in this paper is the weighted Chebyshev constant of $E$ \cite[p. 163]{ST} defined by
\begin{align}\label{wCheb}
t_w := \lim_{m\to\infty} \left( \inf\{ \|w^m Q_m\|_E : Q_m\in\C[z] \mbox{ is monic, } \deg Q_m = m\} \right)^{1/m},
\end{align}
where $\|\cdot\|_E$ denotes the standard supremum norm on $E$, see Chapter III of \cite{ST} for a complete exposition. Note that if $w\equiv 1$ on $E$, then
$t_w$ reduces to the regular Chebyshev constant of $E$ that is equal to the transfinite diameter or capacity of $E$. In the context of our weight
$w(z):=|z|^{-1/2},\ z\in E$, the main application to our problems is that $t_w<1$ implies existence of weighted polynomials $w^m Q_m$ with geometrically
small supremum norms on $E$, hence existence of rational functions $h(z)$ of the form required in Theorem \ref{Rob}. However, completing the proof of our
main result in Theorem \ref{Bconj} via this approach needs a detailed and somewhat technical study by using weighted potential theory (or potential theory
with external fields), which is carried out in the next section.

\section{Technical ingredients} \label{Potential}

This section contains various auxiliary statements necessary to justify all steps of our argument, and complete a proof of Theorem \ref{Bconj}.
The first lemma provides more details on $F$ and its Laurent series used in Theorem \ref{Rob}.

\begin{lemma} \label{expansions}
Let $F$ be as defined in \eqref{fun}. Then $F$ is analytic in $G:=\C\setminus E,$ and satisfies $F(z) = z^n F(1/z),\ z\in G.$
Moreover, $F(z)$ has Laurent expansions with integer coefficients of the form:
\[
F(z)=\sum_{k=0}^\infty b_k z^k \ \mbox{near $0$} \quad \mbox{and} \quad F(z)=\sum_{k=0}^\infty b_k z^{n-k} \ \mbox{near $\infty$}.
\]
\end{lemma}

\begin{proof}
The definition of $F$ and its analyticity in $\C\setminus E$ was discussed in the previous section. Recall that $P(z)=\prod_{k=1}^n (z-\alpha_k)$
is reciprocal, which implies that $P(0) = (-1)^n \prod_{k=1}^n \alpha_k = 1.$ Hence $\prod_{k=1}^n \alpha_k^2 = \prod_{k=1}^n \alpha_k^4 = 1.$
It also follows that the complete set of conjugates $\{\alpha_k\}_{k=1}^n$ is invariant under complex conjugation and inversion in $\T$ given by the
map $z\to1/\bar{z},$ so that this set is invariant under the map $z\to 1/z.$ This further entails similar symmetry properties of $E$ as defined in \eqref{E},
under our convention of symmetric pairing for $\alpha_k$ and $\alpha_{n-k+1}$ by setting $\alpha_k=1/\bar{\alpha}_{n-k+1}$, for $k=1,\ldots,n/2.$ Thus
if $z\in G$ then $1/z\in G$. For $z$ in a neighborhood of the origin, we obtain from the definition of $F$ that
\begin{align*}
z^n  F(1/z) &= z^n \sqrt{\prod_{k=1}^{n} (1/z-\alpha_k^2)(1/z-\alpha_k^4)} = z^n \sqrt{z^{-2n} \prod_{k=1}^{n} \alpha_k^2 \ \prod_{k=1}^{n} \alpha_k^4 \
\prod_{k=1}^{n} (\alpha_k^{-2}-z)(\alpha_k^{-4}-z)} \\
&= z^n z^{-n} \sqrt{\prod_{k=1}^{n} (z-\alpha_k^2)(z-\alpha_k^4)} = F(z),
\end{align*}
where we also used our choice of branch for the square root. Since both $z^n  F(1/z)$ and $F(z)$ are analytic in $G$ and coincide on an open set, they are identical
for all $z\in G.$ The fact that $F$ has Maclaurin expansion with integer coefficients near the origin is immediate from Proposition \ref{Dimprop}. Then the stated Laurent
expansion near infinity follows from the latter Maclaurin expansion by the formula $F(z) = z^n  F(1/z)$ for $z$ near infinity.
\end{proof}

As explained in the end of Section \ref{Outline}, we need to develop a detailed study of the weighted Chebyshev constant $t_w$ for $w(z):=|z|^{-1/2},\ z\in E$.
In the remaining part of this section, we obtain an explicit form of $t_w$ in terms of standard logarithmic potential theory, which allows to find a sharp estimate
expressed through the house of $\alpha.$ We use many facts and ideas from potential theory in the complex plane below, and refer to \cite{Ra} and \cite{Ts} for
the classical version, as well as to \cite{ST} for the weighted version. For a positive unit Borel measure $\mu$ with compact support, define its logarithmic potential by
\[
U^{\mu}(z):=-\dis\int\log|z-t|\,d\mu(t).
\]
The weighted equilibrium measure $\mu_w$ is a unique probability measure supported on $E$ that expresses a steady state distribution of the unit charge in presence
of the external field $Q(z) = - \log w(z),$ where we assume here that $w$ is a positive continuous function defined on $E$. The equilibrium is described by the following
equations for the combined potential of $\mu_w$ and the external field $Q$:
\begin{equation} \label{3.1}
U^{\mu_{w}}(z) + Q(z) \geq F_{w}, \qquad z \in E,
\end{equation}
and
\begin{equation} \label{3.2}
U^{\mu_{w}}(z) + Q(z) = F_{w}, \qquad z \in {\rm supp}\, \mu_{w},
\end{equation}
where $F_{w}$ is a constant (see Theorems I.1.3 and I.5.1 in \cite{ST}). These equations are of importance for us because
the weighted Chebyshev constant is given by
\begin{equation} \label{twFw}
t_w = e^{-F_w}.
\end{equation}
Note that for $w\equiv 1$ on $E$ we have $Q\equiv 0$, so that $\mu_w$ reduces to the classical (not weighted) equilibrium measure $\mu_E$,
and $F_w$ reduces to the classical Robin's constant $V_E$ for $E$. Since logarithmic capacity (and the transfinite diameter) of $E$ is given by
cap$(E)=e^{-V_E},$ the connection of $t_w$ with these classical notions is apparent.

\begin{lemma} \label{twlem}
Let $E\subset\C$ be a compact set with no interior such that $G=\overline{\C}\setminus E$ is connected and $0\in G.$
If $w(z):=|z|^{-1/2},\ z\in E,$ then
\begin{equation} \label{tw}
t_w = e^{-g_G(0,\infty)/2} \sqrt{\textup{cap}(E)}.
\end{equation}
where $g_G(t,\infty)$ is the Green function of $G$ with logarithmic pole at $\infty.$
\end{lemma}

\begin{proof}
Equations \eqref{3.1} and \eqref{3.2} characterize $\mu_w$ in the sense that if for a positive unit Borel measure $\mu$ supported on $E$ one has
\begin{equation} \label{3.5}
U^{\mu}(z) + \frac{1}{2} \log|z| \geq C, \qquad z \in E,
\end{equation}
and
\begin{equation} \label{3.6}
U^{\mu}(z) + \frac{1}{2} \log|z| = C, \qquad z \in S\subset E,
\end{equation}
where $C$ is a constant, then Theorem 3.3 of \cite[Ch. I]{ST}) implies that $\mu_w=\mu$ and $F_w=C.$ We show that the equilibrium measure is given by
\begin{equation}\label{3.7}
\mu := \frac{\omega_G(\infty,\cdot) + \omega_G(0,\cdot)}{2},
\end{equation}
where $\omega_G(\xi,\cdot)$ is the harmonic measure at $\xi\in G$, relative to  $G$. In fact, we verify that \eqref{3.6} holds for all $z\in E$.
The harmonic measure $\omega_G(0,\cdot)$ is the balayage of the point mass $\delta_0$ from the domain $G$ onto its boundary $\partial G=E$,
see Section II.4 of \cite{ST}. It follows from Theorem 4.4 of \cite[p.~115]{ST} that the potential of $\omega_G(0,\cdot)$ satisfies
\begin{align}\label{3.8}
U^{\omega_G(0,\cdot)}(z) + \log|z| &= U^{\omega_G(0,\cdot)}(z) -  U^{\delta_0}(z) = \int g_G(t,\infty)\,d\delta_0(t)= g_G(0,\infty), \quad z\in E.
\end{align}
We recall from Frostman's Theorem that the potential of $\omega_G(\infty,\cdot)=\mu_E$, which is the standard equilibrium measure of $E$,
is equal to Robin's constant on $E$:
\begin{align}  \label{3.9}
U^{\omega_G(\infty,\cdot)}(z)= V_E,\ z\in E,
\end{align}
see, e.g., \cite[p. 59]{Ra}. Combining \eqref{3.8} with \eqref{3.9}, we obtain that
\begin{align*}
U^{\mu}(z) + \frac{1}{2} \log|z| &= \frac{1}{2} \left(U^{\omega_G(0,\cdot)}(z) + \log|z|\right) + \frac{1}{2} U^{\omega_G(\infty,\cdot)}(z) \\
&= \frac{1}{2} g_G(0,\infty) + \frac{1}{2} V_E,\ z\in E,
\end{align*}
and so conclude that \eqref{3.6} is satisfied with the constant given by
\begin{equation} \label{Fw}
F_w = \frac{V_E + g_G(0,\infty)}{2}.
\end{equation}
Hence \eqref{tw} follows from \eqref{Fw} and \eqref{twFw}.
\end{proof}

In the subsequent analysis, we use the notation $t_w(S)$ for the weighted Chebyshev constant of a compact set $S$ with
respect to the weight $w(z):=|z|^{-1/2},\ z\in S.$ Our next goal is to find a convenient upper bound for $t_w$ via symmetrization.

\begin{lemma} \label{symlem}
In the settings of Theorem \ref{Bconj}, let $E$ be as defined in \eqref{E}, and let
\begin{equation} \label{E*}
E^*:=\bigcup_{1\le k\le n} \left( e^{2\pi k i/n} \left[\house{\alpha}^{-4},\house{\alpha}^4\right]\right).
\end{equation}
For the weight $w(z):=|z|^{-1/2},$ we have
\begin{equation} \label{twEE*}
t_w(E) \le t_w(E^*).
\end{equation}
\end{lemma}

\begin{proof}
It is clear from the definition \eqref{E} that $E$ is contained in the closed annulus $\{z\in\C:\house{\alpha}^{-4}\le|z|\le\house{\alpha}^4\}$.
The first step is to construct the set $\tilde{E}$ by extending all radial segments of $E$ so that they connect the circles $|z|=\house{\alpha}^{-4}$
and $|z|=\house{\alpha}^4$. Theorem 5.1.2 of \cite[p. 128]{Ra} gives that $\textup{cap}(E)\le\textup{cap}(\tilde{E})$ because $E\subset\tilde{E}.$
Also, Corollary 4.4.5 of \cite[p. 108]{Ra} gives that $g_G(0,\infty)\ge g_{\tilde{G}}(0,\infty)$ because $\tilde{G}:=\overline{\C}\setminus \tilde{E}
\subset G:=\overline{\C}\setminus E$. It follows from \eqref{tw} that
\begin{equation} \label{3.13}
t_w(E) \le t_w(\tilde{E}).
\end{equation}
Alternatively, one can observe the above inequality directly from the definition of $t_w$ in \eqref{wCheb}, as the supremum norms in that
definition increase with the set.

On the second step, we show that  $t_w$ increases if the radial segments of the set $\tilde{E}$ become equally spaced in angular sense, matching the property
of capacity (transfinite diameter) proved by Dubinin, see Corollary 4.7 of \cite[p. 118]{Dub} and the sketch of proof for Dimitrov's result in Section \ref{Outline}.
In fact, Dubinin's result states that
\begin{equation} \label{3.14}
\textup{cap}(\tilde{E})\le\textup{cap}(E^*).
\end{equation}
In view of \eqref{tw}, it remains to prove that  $g_{\tilde{G}}(0,\infty)\ge g_{G^*}(0,\infty)$, where $G^*:=\overline{\C}\setminus E^*$. We deduce this
fact from Theorem A of Solynin \cite[p. 1702]{So} that states a corresponding result for harmonic measures:
\begin{equation} \label{3.15}
\omega_{\tilde{D}_R}(0,\tilde{E}) \le \omega_{D^*_R}(0,E^*),
\end{equation}
where $\omega_{\tilde{D}_R}(0,\tilde{E})$ is the harmonic measure of $\tilde{E}$ at $0$, relative to $\tilde{D}_R:=\{z\in\C:|z| < R\}\setminus\tilde{E}$,
and $\omega_{D^*_R}(0,E^*)$ is the harmonic measure of $E^*$ at $0$, relative to $D^*_R:=\{z\in\C:|z| < R\}\setminus E^*$, for sufficiently large $R$.
We note that the result of Solynin in \cite{So} is stated for the unit disk, i.e., for $R=1,$ but \eqref{3.15} immediately follows by using the scaling map $z\to Rz.$
Recall that $\omega_{\tilde{D}_R}(z,\tilde{E})$ is a harmonic function of $z$ in $\tilde{D}_R,$ which is continuous on the closure of this domain, and has boundary
values $\omega_{\tilde{D}_R}(z,\tilde{E})=1,\ z\in\tilde{E},$ and $\omega_{\tilde{D}_R}(z,\tilde{E})=0,\ |z|=R.$ Hence the function
\begin{equation} \label{3.16}
\tilde{u}_R(z) := (1-\omega_{\tilde{D}_R}(z,\tilde{E}))(\log{R}+V_{\tilde{E}})
\end{equation}
is harmonic in $\tilde{D}_R,$ continuous on the closure of this domain, and has boundary values $\tilde{u}_R(z)=0,\ z\in\tilde{E},$ and
$\tilde{u}_R(z)=\log{R}+V_{\tilde{E}},\ |z|=R,$ where $V_{\tilde{E}}$ denotes the Robin's constant of $\tilde{E}.$ Since the Green function $g_{\tilde{G}}(z,\infty)$
has similar properties of being harmonic in $\tilde{G}\setminus\{\infty\}$, with boundary value $0$ on $\tilde{E}$ and the asymptotic
$g_{\tilde{G}}(z,\infty)=\log{R}+V_{\tilde{E}}+o(1)$ as $|z|=R\to\infty,$ we obtain from the Maximum-Minimum Principle that
\[
|\tilde{u}_R(z) - g_{\tilde{G}}(z,\infty)| \le o(1), \quad z\in \tilde{D}_R,\ R\to\infty.
\]
Setting $R=N\in\N$, we produce a sequence of harmonic functions $\tilde{u}_N(z)$ that converges to $g_{\tilde{G}}(z,\infty)$ uniformly on compact subsets of $\C$
as $N\to\infty.$ The same construction yields the sequence of harmonic functions
\begin{equation*}
u^*_N(z) := (1-\omega_{D^*_R}(z,E^*))(\log{N}+V_{E^*})
\end{equation*}
that converges to $g_{G^*}(z,\infty)$ uniformly on compact subsets of $\C$. Applying \eqref{3.15} and the equivalent of \eqref{3.14} written as
$V_{\tilde{E}} \ge V_{E^*}$, we conclude that
\[
\tilde{u}_N(0) \ge u^*_N(0) \quad \mbox{for large $N\in\N$.}
\]
Passing to the limit as $N\to\infty$, we arrive at $g_{\tilde{G}}(0,\infty)\ge g_{G^*}(0,\infty),$ which implies that
\[
t_w(\tilde{E}) \le t_w(E^*)
\]
by \eqref{3.14} and \eqref{tw}. Thus \eqref{twEE*} follows from the latter inequality and \eqref{3.13}.
\end{proof}

We are now ready to find an explicit bound for $t_w$ expressed via the house of $\alpha.$
\begin{lemma} \label{seglem}
In the settings of Theorem \ref{Bconj}, let $E^*$ be as defined in \eqref{E*} and let $I:=\left[\house{\alpha}^{-4n},\house{\alpha}^{4n}\right]\subset\R$.
For the weight $w(z):=|z|^{-1/2},$ we have
\begin{equation} \label{twE*I}
t_w(E^*) \le \left(t_w\left(I\right)\right)^{1/n},
\end{equation}
where
\begin{equation} \label{twseg}
t_w\left(I\right) = \frac{\house{\alpha}^{2n}-\house{\alpha}^{-2n}}{2}.
\end{equation}
\end{lemma}

\begin{proof}
We use the definition of $t_w(I)$ given in \eqref{wCheb}, first replacing $E$ with $I$ in that equation. Let $Q_m$ be a sequence of monic
polynomials that realize the $\inf$ in \eqref{wCheb} for $t_w(I)$. For $z\in E^*$, we have that $t=z^n\in I$ and
\begin{align*}
\left(t_w\left(I\right)\right)^{1/n} = \lim_{m\to\infty} \left\| |t|^{-m/2} Q_m(t) \right\|_I^\frac{1}{mn} \ge
\liminf_{m\to\infty} \left\| |z|^{-mn/2} Q_m(z^n) \right\|_{E^*}^\frac{1}{mn} \ge t_w(E^*).
\end{align*}
In fact, equality holds in \eqref{twE*I}, but the stated inequality is sufficient for our purpose. We now compute $t_w\left(I\right)$ from \eqref{tw}.
Corollary 5.2.4 of \cite[p. 134]{Ra} immediately gives that
\begin{equation} \label{3.19}
\textup{cap}\left(I\right) = \left(\house{\alpha}^{4n}-\house{\alpha}^{-4n}\right)/4.
\end{equation}
We now need to find the value of $g_\Omega(0,\infty)$, where $\Omega:=\C\setminus I.$ This is conveniently available from the well known connection
with the conformal mapping $\Phi:\Omega\to\{w:|w|>1\}$ normalized by $\Phi(\infty)=\infty:$
\begin{equation*}
g_\Omega(z,\infty) = \log|\Phi(z)|, \quad z\in\Omega,
\end{equation*}
see, e.g., the proof of Theorem 4.4.1 in \cite[p. 113]{Ra}. If $I=[a,b]\subset\R$ then $\Phi$ is given explicitly by the (shifted and scaled) inverse of
the Joukowski conformal mapping:
\begin{equation*}
\Phi_{\infty} (z) := \frac{2z -a -b +2 \sqrt{(z-a)(z-b)}}{b-a}, \quad z\in\Omega.
\end{equation*}
Applying this with $a=\house{\alpha}^{-4n}$ and $b=\house{\alpha}^{4n}$, we obtain that
\[
e^{-g_\Omega(0,\infty)} = 1/|\Phi(0)| = \frac{\house{\alpha}^{4n}-\house{\alpha}^{-4n}}{\house{\alpha}^{4n}+\house{\alpha}^{-4n}+2}.
\]
The latter formula together with \eqref{3.19} and \eqref{tw} yield
\begin{align*}
t_w(I) = \frac{\house{\alpha}^{4n}-\house{\alpha}^{-4n}}{2(\house{\alpha}^{2n}+\house{\alpha}^{-2n})} =  \frac{\house{\alpha}^{2n}-\house{\alpha}^{-2n}}{2}.
\end{align*}
\end{proof}

Our last lemma gives information about constant terms of polynomials with asymptotically minimal weighted norms,
which is necessary for the application of Theorem \ref{Rob}.

\begin{lemma} \label{constlem}
Let $E$ be as defined in \eqref{E}, and let $w(z)=|z|^{-1/2},\ z\in E$. Then there is a sequence of monic polynomials
$Q_m,\ \deg(Q_m)=m,$ that satisfies
\begin{align} \label{wnorm}
t_w(E) = \lim_{m\to\infty} \||z|^{-m/2} Q_m(z)\|_E^{1/m}
\end{align}
and
\begin{equation} \label{const}
\lim_{m\to\infty} |Q_m(0)|^{1/m} = 1.
\end{equation}
\end{lemma}

\begin{proof}
The required sequence of polynomials is selected as the weighted Fekete polynomials introduced and studied in Section III.1
of \cite{ST}. In particular, Theorem 1.9 of \cite[p. 150]{ST} states that the weighted Fekete polynomials $Q_m$ associated with the weight
$w(z)=|z|^{-1/2},\ z\in E,$ satisfy \eqref{wnorm}, and Theorem 1.8 of \cite[p. 150]{ST} states that
\begin{equation*}
\lim_{m\to\infty} |Q_m(z)|^{1/m} = \exp\left(-U^{\mu_{w}}(z)\right),\quad z\in G,
\end{equation*}
where $\mu_w$ is the weighted equilibrium measure found in the proof of Lemma \ref{twlem}, and explicitly given in \eqref{3.7}.
Hence \eqref{const} is equivalent to $U^{\mu_{w}}(0)=0,$ which further translates into
\begin{align}\label{3.21}
U^{\omega_G(\infty,\cdot)}(0) + U^{\omega_G(0,\cdot)}(0) = 0.
\end{align}
For the classical equilibrium (conductor) potential $U^{\omega_G(\infty,\cdot)}$, we have
\begin{align} \label{3.22}
U^{\omega_G(\infty,\cdot)}(z) = V_E - g_G(z,\infty),\ z\in G,
\end{align}
by Theorem III.37 in \cite[p.~82]{Ts}. We also use that $\omega_G(0,\cdot)$ is the balayage of $\delta_0$ from $G$, as in the proof of Lemma \ref{twlem}.
It follows from Theorem 5.1 of \cite[p. 124]{ST} that
\begin{align*}
U^{\delta_0}(z) - \int g_G(z,t)\,d\delta_0(t) = U^{\omega_G(0,\cdot)}(z) - \int g_G(t,\infty)\,d\delta_0(t), \quad z\in G,
\end{align*}
consequently,
\begin{align} \label{3.23}
U^{\omega_G(0,\cdot)}(z) = g_G(0,\infty) - g_G(z,0) - \log|z|, \quad z\in G.
\end{align}
Since $G$ is invariant under the transformation $w=1/z$, we also have the connection \cite[p. 108]{Ra}
\[
g_G(z,0) = g_G(1/z,\infty),\quad z\in G.
\]
Combining \eqref{3.22} and \eqref{3.23} with the latter fact, we verify \eqref{3.21} as follows:
\begin{align*}
U^{\omega_G(\infty,\cdot)}(0) + U^{\omega_G(0,\cdot)}(0) &= V_E - g_G(0,\infty) + g_G(0,\infty) - \lim_{z\to 0} \left(g_G(z,0) + \log|z|\right) \\
&= V_E - \lim_{w\to \infty} \left(g_G(w,\infty) - \log|w|\right) = V_E - V_E = 0,
\end{align*}
where we used \eqref{3.22} and Theorem 3.1.2 of \cite[p. 53]{Ra} to compute the last limit.
\end{proof}

\section{Proof of the main result} \label{Proof}

\begin{proof}[Proof of Theorem \ref{Bconj}]
The lower bound for the house of $\alpha$ in \eqref{c-conj} follows from the inequality
\begin{equation} \label{twineq}
\frac{\house{\alpha}^{2n}-\house{\alpha}^{-2n}}{2} \ge 1,
\end{equation}
which turns into the quadratic inequality $x^2-2x-1\ge 0$ after the substitution $x=\house{\alpha}^{2n}.$ We show that \eqref{twineq} holds
by contradiction, following the argument sketched in Section \ref{Outline}. Indeed, if we assume that the opposite of \eqref{twineq}
holds, then Lemmas \ref{symlem} and \ref{seglem} give that
\begin{equation*}
t_w(E) \le t_w(E^*) \le  \left(t_w\left(I\right)\right)^{1/n} = \left( \frac{\house{\alpha}^{2n}-\house{\alpha}^{-2n}}{2} \right)^{1/n} < 1.
\end{equation*}
Applying Lemma \ref{constlem}, we find a sequence of monic polynomials $Q_m,\ \deg(Q_m)=m,$ that satisfies
\begin{align*}
\lim_{m\to\infty} \||z|^{-m/2} Q_m(z)\|_E^{1/m} < 1
\end{align*}
and \eqref{const}. For $m=2k,\ k\in\N,$ the weighted polynomials $z^{-m/2} Q_m(z)$ take the following form:
\[
h_k(z) := z^{-k} Q_{2k}(z) = A_{-k,k} z^{-k} + \sum_{j=-k+1}^{k-1} A_{j,k} z^j + z^k,
\]
with $|h_k(z)|<1,\ z\in E,$ for all large $k\in\N.$ If $|A_{-k,k}|\ge 1$ for some sufficiently large $k$, then the corresponding $h_k$
can be used directly as $h(z)$ in Theorem \ref{Rob}. However, it may happen that $|A_{-k,k}| < 1$ for all large $k\in\N$, in which case we consider
\[
g_k(z) := h_k(z)/A_{-k,k} = z^{-k} + \sum_{j=-k+1}^{k-1} A_{j,k}z^j/A_{-k,k} + z^k/A_{-k,k},\quad k\in\N\mbox{ is large},
\]
that satisfies the requirements of Theorem \ref{Rob} for $h(z)$ because $A_{-k,k}=Q_{2k}(0)$, and \eqref{const} gives
\begin{align*}
\lim_{k\to\infty} \|g_k(z)\|_E^{1/k} = \lim_{k\to\infty} \|z^{-k} Q_{2k}(z)/A_{-k,k} \|_E^{1/k} = \lim_{k\to\infty} \|z^{-k} Q_{2k}(z)\|_E^{1/k} < 1.
\end{align*}
We would like to apply Theorem \ref{Rob} to our function $F(z)$ defined in \eqref{fun}, which has Laurent expansions with integer coefficients
at $\infty$ and at $0$ according to Lemma \ref{expansions}. Note however that near $\infty$
\[
F(z)=\sum_{k=0}^\infty b_k z^{n-k},
\]
which formally does not fit the assumptions on this expansion in Theorem \ref{Rob}, as it contains $n$ positive powers of $z$. The latter problem
is easily remedied by subtracting the polynomial part of this expansion from $F$, and instead applying Theorem \ref{Rob} to
\[
F(z) - \sum_{k=0}^n b_k z^{n-k}.
\]
The conclusion of Theorem \ref{Rob} is that the above function is rational, hence $F$ is rational too. But Proposition \ref{Dimprop} implies now
that $\alpha$ is a root of unity, which is an obvious contradiction to our original assumptions of Theorem \ref{Bconj}, provided that $P_2$ is not a
perfect square.

We now use induction to handle the remaining case when $P_2$ is a perfect square. The basis is easily established by considering quadratic reciprocal polynomials
$P(z)=z^2+bz+1,\ b\in\Z,$ where $|b|\ge 3$ because $P$ cannot have roots on $\T$. For a root $\alpha$ of such $P$, we clearly have
\[
\house{\alpha} = \frac{|b|+\sqrt{b^2-4}}{2} \ge \frac{3+\sqrt{5}}{2} > (1+\sqrt{2})^{1/4},
\]
confirming \eqref{c-conj} for $n=2.$ Assume now that \eqref{c-conj} holds for all $\alpha$ of degree less than $n,$ satisfying the assumptions of this theorem.
Let $P_2(z)=(R(z))^2,$ where $R\in\Z[z]$ is a monic polynomial with roots from the set $\{\alpha_k^2\}_{k=1}^n$, which is symmetric with respect to $\T$.
As $P_2$ has double roots according to our assumption, the set $\{\alpha_k^2\}_{k=1}^n$ is composed of pairs $\alpha_j^2=\alpha_k^2,\ j\neq k.$ Hence
we can assume that $\alpha_{n/2+k} = -\alpha_k,\ k=1,\ldots,n/2,$ after a proper rearrangement. Thus $R(z)=\prod_{k=1}^{n/2} (z-\alpha_k^2)$, and it
inherits the reciprocal property from $P_2.$  Since $P(z)=\prod_{k=1}^{n} (z-\alpha_k)=\prod_{k=1}^{n/2} (z^2-\alpha_k^2)=R(z^2)$, where $P$ is
irreducible, we conclude that $R$ is also irreducible. Letting $\house{\alpha}=|\alpha_M|,\ 1\le M\le n/2,$ we obtain that $R$ is the minimal polynomial of
$\alpha_M^2$ of degree $n/2.$  The induction hypothesis implies that
\[
|\alpha_M|^2 = \house{\alpha_M^2} \ge (1+\sqrt{2})^{\frac{1}{n}},
\]
which gives
\[
\house{\alpha} = |\alpha_M| \ge (1+\sqrt{2})^{\frac{1}{2n}}
\]
as required.

\end{proof}

\medskip
{\bf Acknowledgements.}
This research was partially supported by NSF via the American Institute of Mathematics, and by the Vaughn Foundation
endowed Professorship in Number Theory. The author would like to thank Art\=uras Dubickas for several helpful discussions and
references.

\end{document}